\documentclass[11pt]{amsart} %

\usepackage{float}

\usepackage{epsfig}

\usepackage{array, amsmath, enumerate, url, psfrag}
\usepackage{amssymb, amsaddr, fullpage}
\usepackage{graphicx,subfigure}
\usepackage[usenames,dvipsnames]{pstricks}
\usepackage{color}

\restylefloat{figure}
 \makeatother

\newtheorem{theorem}{Theorem}[section]
\newtheorem{lemma}[theorem]{Lemma}
\newtheorem{conjecture}[theorem]{Conjecture}

\newtheorem{definition}{Definition}

\title{$(2,4)$-Colorability of Planar Graphs Excluding $3$-, $4$-, and $6$-Cycles}
\author{ Pongpat Sittitrai$^{1}$ \hskip 0.2in Wannapol Pimpasalee$^{2}$\hskip 0.2in Kittikorn Nakprasit$^{3, 4, *}$}

\address{
	$^{1}$\small Futuristic Science Research Center, School of Science, Walailak University, Nakhon Si Thammarat 80160, Thailand. \newline
	Email: pongpat.sittitrai@gmail.com
	\newline
	$^{2}$\small Department of Science and Mathematics, Faculty of Science and Health Technology, Kalasin University, Kalasin, 46000, Thailand.\newline
	Email: pwannapol@yahoo.com \newline
	$^{3}$\small Department of Mathematics, Faculty of Science, Khon Kaen University, Khon Kaen, 40002, Thailand.\newline
	$^{4}$\small Centre of Excellence in Mathematics, MHESI, Bangkok 10400, Thailand.\newline
	Email: kitnak@hotmail.com
}
\begin{document}
	
	\maketitle
	
	\begin{center}{\bf Abstract}\end{center}
	\indent\indent

	A defective $k$-coloring is a coloring on the vertices of a graph using colors $1,2, \dots, k$ such that adjacent vertices may share the same color. 
	A $(d_1,d_2)$-\emph{coloring} of a graph $G$ is a defective $2$-coloring of $G$ such that any vertex colored by color $i$ has at most $d_i$ adjacent vertices of the same color, where $i\in\{1,2\}$. 
	A graph $G$ is said to be  $(d_1,d_2)$-\emph{colorable} if it admits a $(d_1,d_2)$-coloring.

	Defective $2$-coloring in planar graphs without $3$-cycles, $4$-cycles, and $6$-cycles has been investigated by Dross and Ochem, as well as Sittitrai and Pimpasalee.  
	They showed that such graphs are $(0,6)$-colorable and $(3,3)$-colorable, respectively. 
	In this paper, we proved that these graphs are also $(2,4)$-colorable.

	\section{Introduction} 
	
	The concept of defective coloring (or improper coloring) was independently introduced by Andrews and Jacobson \cite{conceptdefective2}, Cowen et al. \cite{c32}, and Harary and Jones \cite{conceptdefective3}. 
	Let $d_1, d_2, \ldots, d_k$ be $k$ nonnegative integers.
	A $(d_1, d_2, \ldots, d_k)$-\emph{coloring} of a graph $G$ is a coloring 
	on the vertices of $G$ using colors $1,2, \dots, k$ such that any vertex assigned color $i$ has at most $d_i$ adjacent vertices of the same color, where $i\in\{1,2, \dots, k\}$. 
	If a graph $G$ admits a $(d_1, d_2, \ldots, d_k)$-coloring, then a graph $G$ is a $(d_1, d_2, \ldots, d_k)$-\emph{colorable}.
	
	Defective coloring generalizes proper coloring, as any $k$-colorable graph is a $(d_1, d_2, \ldots, d_k)$-colorable graph with $d_1 = d_2 = \cdots = d_k = 0$. 
	For example, The Four Color Theorem, proven by Appel and Haken \cite{app1, app2}, can be restated as every planar graph is $(0,0,0,0)$-colorable.

	There are several sufficient conditions for a planar graph to be $3$-colorable, or equivalently $(0,0,0)$-colorable, such as every planar graph without $3$-cycles is $3$-colorable \cite{g3f}.
	In 1976, Steinberg conjectured that every planar graph without $4$- and $5$-cycles is $3$-colorable, (open problem 2.9 in Jensen and Toft \cite{32}).
	Although this conjecture was later disproved by Cohen-Addad et al. \cite{add}, many researchers have continued to explore results related to Steinberg’s conjecture through defective coloring \cite{che200, xu110}.
	Similarly, several works attempt to find $(d_1, d_2, d_3)$-colorable planar graphs without cycles of certain lengths e.g., \cite{c32, c47200, c478110, c46200}.
	
	Inspired by Steinberg's conjecture,  Sittitrai and Nakprasit \cite{C452defec} initiated the study of $(d_1, d_2)$-coloring on planar graphs without $4$- and $5$-cycles. 
	They showed that these graphs are $(2,9)$-, $(3,5)$-, and $(4,4)$-colorable.
	Subsequently, some researchers have improved these results. 
	Liu and Lv \cite{Liu26} showed that these graphs are $(2,6)$-colorable. 
	Recently, Liu and Xiao \cite{Liu33} and  Li et al. \cite{Li33} independently proved that these graphs are also $(3,3)$-colorable.

	Additionally, Ma et al. \cite{Ma29} showed that planar graphs without $4$- and $6$-cycles are $(2,9)$-colorable, while Nakprasit et al. \cite{c4634} showed that they are also $(3,4)$-colorable. 
	
	For planar graphs without three specific short cycles, Borodin and Kostochka \cite{bo2}  showed that planar graphs without $3$-, $4$-, and $5$-cycles (equivalently, planar graphs with girth at least $6$) are $(1,4)$-colorable, while Havet and Sereni \cite{h44} showed that these graphs are $(2,2)$-colorable. 
	Dross and Ochem \cite{Dross06} proved that planar graphs without $3$-, $4$-, and $6$-cycles are $(0,6)$-colorable. 
	Recently, Sittitrai and Pimpasalee \cite{C34633} showed that such graphs are also $(3,3)$-colorable and, inspired by these results, posed the following conjecture.
	
	\begin{conjecture}\label{prob}
		Every planar graph without $3$-, $4$-, and $6$-cycles is $(d_1,d_2)$-colorable if $d_1+d_2 \geq 6$. 
	\end{conjecture}
	
	We partially verify the conjecture with the following theorem. 
	\begin{theorem}\label{main}
		Every planar graph without $3$-, $4$-, and $6$-cycles is $(2,4)$-colorable. 
	\end{theorem}

	Notation and definitions used throughout this work are introduced as follows. 
	A \emph{$k$-vertex} denotes a vertex with a degree of exactly $k$, while a \emph{$k^+$-vertex} and a \emph{$k^-$-vertex} refer to vertices with a degree of at least $k$ and at most $k$, respectively. 
	This notation extends similarly to faces.
	
	The boundary walk associated with a face $f$ is denoted by $b(f)$. 
	A vertex $v$ is said to be \emph{incident} to a face $f$ if it lies on the boundary walk $b(f)$. 
	We let $N_k(f)$ denote the number of incident $k$-vertices of a face $f$.
	
	\section{Structures of Minimal Counterexamples} 
	
	This section presents the properties of a minimal counterexample $G$ to Theorem~\ref{main}, which are essential for proving the main result.
	A minimal counterexample $G$ has neither $3$-, $4$-, nor $6$-cycles and is not $(2,4)$-colorable, but each proper subgraph of $G$ is $(2,4)$-colorable.
	It is clear that $G$ is connected, and each vertex in $G$ has a degree of at least $2$.
	To be precise, each of the following results only assumes the necessary conditions of $G$ to have the desired properties.

	\begin{lemma}\label{l3} (Lemmas 2.1 and 2.2 in \cite{choi35})
		Let $G$ be a minimal non-$(d_1,d_2)$-colorable graph where $d_1 \leq d_2$.  
		If $v$ is a $3^-$-vertex in $G$, then $v$ is adjacent to at least two $(d_1+2)^+$-vertices, one of which is a $(d_2+2)^+$-vertex.
	\end{lemma}

	\begin{lemma}\label{l1}  
		Let $v$ be a vertex of a minimal non-$(2,4)$-colorable graph $G.$ 
		\begin{enumerate}
			\item [\rm (i)] If $v$ is a $3^-$-vertex, then $v$ is adjacent to two $4^+$-vertices, one of which is  a $6^+$-vertex. 
			\item [\rm (ii)]  If $v$ is a $4$-vertex, then $v$ is adjacent to a $6^+$-vertex. 
		\end{enumerate}
	\end{lemma}

	\begin{proof} 
		(i) It follows directly from Lemma \ref{l3}. 
		
		(ii) Let $v$ be a $4$-vertex of $G$. 
		Suppose to the contrary that $v$ is adjacent to four $5^-$-vertices, say $v_1,v_2,v_3,$ and $v_4$. 
		
		By minimality of $G$, the graph $G - v$ has a $(2, 4)$-coloring. 
		Note that each $v_i$ is a $4^-$-vertex in $G - v$ for $i\in\{1,2,3,4\}$. 
		If $v_i$ is colored with $2$, and all adjacent vertices of $v_i$ in $G-v$ are also colored by $2$, then we recolor $v_i$ to color $1$. 
		
		One can see the new coloring is also a  $(2, 4)$-coloring of $G-v$ where each $v_i$ colored with color $2$ has to at most three adjacent vertices colored with color $2$. 						
		We can then extend this $(2,4)$-coloring to $G$ by assigning $2$ to $v.$
		
		This contradiction completes the proof. 
	\end{proof}

	\begin{lemma}\label{face} 
		Let $G$ be a minimal counterexample $G$ to Theorem~\ref{main}. 
		Then $G$ has the following properties. 
		\begin{enumerate}
			\item [\rm (i)] If $f$ is a $k$-face where $3\leq k\leq 9$, then $b(f)$ is a $k$-cycle.  
			
			\item [\rm (ii)] $G$ does not contain a $k$-face where $k\in\{3,4,6\}$.
			
			\item [\rm (iii)] A $2$-vertex $v$ is not incident to two $5$-faces. 		
		\end{enumerate}	
	\end{lemma}	
	
	\begin{proof}
		Lemmas \ref{face}(i) and (ii) follow directly from the fact that $G$ contains no $1$-vertices, nor  $3$-, $4$-, or $6$-cycles. 
		
		By combining this fact with Lemma \ref{l1}(i), we obtain Lemma \ref{face}(iii). 
	\end{proof}
	
	\begin{lemma}\label{5face} 
		Let $G$ be a minimal counterexample $G$ to Theorem~\ref{main}. 
		Then each $k$-face $f$ in $G$ has the following properties. 
		\begin{enumerate}
			\item [\rm (i)] $N_2(f)\leq \lfloor\frac{k}{2}\rfloor$.
			
			\item [\rm (ii)] $N_3(f)\leq k-2N_2(f)-1$ if $N_2(f)< \lfloor\frac{k}{2}\rfloor$.
			
			\item [\rm (iii)] $N_3(f)=0$ if $N_2(f)= \lfloor\frac{k}{2}\rfloor$.
		\end{enumerate}	
	\end{lemma}	
	
	\begin{proof} 
		Let $f$ be a $k$-face of $G$ with $b(f)=v_1v_2\dots v_k$.
		
		Define the following sets where each subscript is taken modulo $k.$:
		
		$A$ := the set of $v_i$ where $v_i$ is a $2$-vertex. 
		
		$B$ := the set of $v_i$ where $v_i$ is a $3$-vertex. 
		
		$C$ := the set of $v_i$ where $v_{i+1}$ is a $2$-vertex. 
		
		$D$ := the set of $v_i$ where $v_{i+1}$ is a $3$-vertex.
		
		By Lemma \ref{l1}(i), the sets $A, B, C,$ and $D$ are pairwise disjoint. 
		Note that $|A|=|C|= N_2(f)$ and $|B|= |D| = N_3(f).$
		The properties in Lemma~\ref{5face} follow directly from these observations.
	\end{proof}

	\section{Proof of Theorem \ref{main}}\label{discharge}
	
	\indent 
	Consider $G$ as a minimal counterexample to Theorem \ref{main}. 
	The discharging process is defined as follows:
	
	Each vertex $v$ in $G$ is assigned an initial charge $\mu(v) = d(v) - 4$, where $d(v)$ denotes the degree of vertex $v$. 
	Similarly, each face $f$ in $G$ is assigned an initial charge $\mu(f) = d(f) - 4$, where $d(f)$ denotes the degree of face $f$.
	
	Using Euler's formula, which states that $|V(G)| - |E(G)| + |F(G)| = 2$, along with the Handshaking lemma, we obtain the following equation:
	
	$$\displaystyle\sum_{v\in V(G)}\mu(v)
	+\displaystyle\sum_{f\in F(G)}\mu(f)=-8.$$
	
	Next, we define a new charge $\mu^*(x)$ for all $x \in V(G) \cup F(G)$ by redistributing the charge among vertices and faces. 
	The total sum of the new charges $\mu^*(x)$ remains $-8$. 
	If the final charge $\mu^*(x) \geq 0$ for all $x \in V(G) \cup F(G)$, this leads to a contradiction, thereby completing the proof.
	
	The rules for discharging are as follows:
	
	\noindent (R1) A $4$-vertex $v$ give charge $t$ to each of its adjacent $2$-vertices where $t=\frac{1}{9}$ if $v$ is adjacent to exactly one $4^+$-vertex, and $t=\frac{1}{6}$ otherwise.	
	
	\noindent (R2) A $5$-vertex $v$ gives charge $\frac{1}{5}$ to each of its adjacent $3^-$-vertices.
	
	\noindent (R3) A $6^+$-vertex $v$ gives charge $t$ to each of its adjacent $4^-$-vertices where $t=\frac{1}{3}$ if $v$ is not adjacent to a $5^+$-vertex, and $t=\frac{2}{5}$ otherwise.
	
	\noindent (R4) A $5$-face $f$ gives charge $t$  to each of its incident $2$-vertices where $t=\frac{5}{9}$ if $f$ is incident to one $2$-vertex, and $t=\frac{1}{2}$ otherwise. 
	
	\noindent (R5) A $7^+$-face $f$ gives charge $1$  to each of its incident $2$-vertices.
	
	\noindent (R6) After applying rules (R1) to (R5), redistribute the charges between any two $2$-vertices that are incident to the same $5$-face so that they each have equal charge.
	
	\noindent (R7) A $5^+$-face $f$ gives charge $\frac{2}{9}$ to each of its incident $3$-vertices.

	It remains to show that the resulting charge $\mu^*(x) \geq 0$ for all $x \in V(G) \cup F(G)$. 
	It is clear that $\mu^*(x) \geq 0$ holds when $x$ is a $k$-face when $x\in\{3,4,6\}$ or when $x$ is a $4$-vertex.

	
	Now, consider a $k$-vertex $v$. 
	
	(1) Suppose $v$ is a $2$-vertex not incident to a $5$-face. 
	It follows from Lemma \ref{face} that $v$ is incident to two $7^+$-face. 
	Then $\mu^*(v) \geq-2+2\times 1=0$ by (R5).\\

	(2) Suppose $v$ is a $2$-vertex incident to a $5$-face $f$ 
	where $f$ is incident to exactly one $2$-vertex. 
	By Lemma \ref{face}, the other incident face of $v$ is a $7^+$-face. 
	Then $\mu^*(v) \geq-2+\frac{1}{9}+\frac{1}{3}+\frac{5}{9}+1=0$ by (R1)-(R5).\\ 
	
	(3) Suppose $v$ is a $2$-vertex incident to a $5$-face $f$ 
	where $f$ is incident to two $2$-vertices.
	Let $b(f)=vv_1v_2uv_3$. 
	By Lemma \ref{l3}, we can suppose that $u$ is a $2$-vertex and each of  $v_1, v_2, v_3$ is a $4^+$-vertex. 
	By Lemma \ref{face}, the other incident face of $v$ and $u$ is a $7^+$-face. 
	
	By (R6), we consider $\mu^*(v)+\mu^*(u)$ instead of only $\mu^*(v)$ in the following three subcases. 
	
	\begin{enumerate}
		\item [\rm (i)] If $v_3$ is a $4$- or $5$-vertex, then by Lemma \ref{l1}(i), both of $v_1$ and $v_2$ are $6^+$-vertices.         
		Since each of $v_1$ and $v_2$ is adjacent to a $6^+$-vertex, it follows that $\mu^*(v)+\mu^*(u) \geq 2\times (-2 + \frac{1}{9} + \frac{2}{5}+\frac{1}{2}+1) >0$ by (R1)-(R5). 
		
		\item [\rm (ii)]  If $v_3$ is a $6^+$-vertex and both of $v_1$ and $v_2$ are $4$- or $5$-vertices, then by Lemma \ref{l1}(ii), each of $v_1$ and $v_2$ is a $5$-vertex or is adjacent to at least two $4^+$-vertices. 
		Thus $\mu^*(v)+\mu^*(u) \geq 2\times (-2+\frac{1}{6} + \frac{1}{3}+\frac{1}{2}+1)=0$ by (R1)-(R5). 
		
		\item [\rm (iii)] If $v_3$ is a $6^+$-vertex and  $v_1$ or $v_2$ is a $6^+$-vertex, then $\mu^*(v)+\mu^*(u) > -4 +3\times \frac{1}{3} +2\times \frac{1}{2} +2\times 1=0$ by (R3)-(R5). 
	\end{enumerate}
	
	Using (R6), we conclude that $\mu^*(v) \geq 0.$
	
	(4)  Suppose $v$ is a $3$-vertex. 
	By Lemma \ref{l1}(i), the vertex $v$ is adjacent to a $6^+$-vertex.  
	Then $\mu^*(v) \geq-1 +\frac{1}{3} +3 \times\frac{2}{9}=0$ by (R2), (R5), and (R7). \\ 
	
	(5) Suppose $v$ is a $4$-vertex. 
	By Lemma \ref{l1}(ii), the vertex $v$ is adjacent to a $6^+$-vertex. 
	If $v$ is adjacent to exactly one $4^+$-vertex, then $\mu^*(v) \geq 0-3\times\frac{1}{9}+\frac{1}{3}=0$, otherwise $\mu^*(v) \geq 0-2\times\frac{1}{6} + \frac{1}{3}=0$ by (R1) and (R3).\\

	(6) Suppose $v$ is a $5$-vertex.   
	Then $\mu^*(v) \geq 1-5\times\frac{1}{5}=0$ by (R2).\\

	(7) Suppose $v$ is a $k$-vertex where $k\geq 6$.  
	If $v$ is adjacent to a $6^+$-vertex, then $\mu^*(v) \geq (k-4)-(k-1)\times\frac{2}{5}\geq 0$, otherwise $\mu^*(v) \geq (k-4)-k\times\frac{1}{6} > 0$ by (R3).\\
	

		Next, we consider a $k$-face $f$. 
		Recall that $N_t(f)$ denotes the number of incident $t$-vertices of $f$.\\
		
		(8) Suppose $f$ is a $5$-face. 
		It follows from Lemma \ref{5face}(i) that $N_2(f)\leq 2$. 
		
		If $N_2(f)=0$, then $N_3(f)\leq 4$ by Lemma \ref{5face}(ii), and thus $\mu^*(f)\geq1-4\times\frac{2}{9}>0$ by (R7). 
		
		If $N_2(f)=1$, then $N_3(f)\leq 2$ by Lemma \ref{5face}(ii), and thus $\mu^*(f)\geq 1-\frac{5}{9}-2\times\frac{2}{9}=0$ by (R4) and (R7). 
		
		If $N_2(f)=2$, then $N_3(f)=0$ by Lemma \ref{5face}(iii), and thus $\mu^*(f)=1-2\times\frac{1}{2}=0$ by (R4).\\

			
			
		
		%
		%

		(9) Suppose $f$ is a $k$-face where $k\geq 7$ and $N_2(f)= \lfloor\frac{k}{2}\rfloor$. 
		It follows from Lemma \ref{5face}(iii) that $N_3(f)=0$. 
		Thus $\mu^*(f)=(k-4)-\lfloor\frac{k}{2}\rfloor\times 1\geq 0$ by (R5).\\ 
		
		
		(10) Suppose $f$ is a $k$-face where $k\geq 7$ and $N_2(f) \neq \lfloor\frac{k}{2}\rfloor$. 
		
		It follows from Lemma \ref{5face}(i) that $N_2(f) < \lfloor\frac{k}{2}\rfloor$. 
		Using  Lemma \ref{5face}(ii), we obtain that $N_3(f)\leq k-2 N_2(f)-1$.
		
		
		By (R5) and (R7), we have 
		\begin{align*}
			\mu^*(f)&=(k-4)-N_2(f)\times1-N_3(f)\times\frac{2}{9}\\
			&\geq (k-4)-N_2(f)\times1-(k-2\times N_2(f)-1)\times\frac{2}{9} \\
			&=\frac{7}{9} k-\frac{34}{9}-\frac{5}{9} N_2(f). 
		\end{align*}
		
		If $k=7,$ then $ \mu^*(f) \geq \frac{49}{9} -\frac{34}{9}-\frac{10}{9}>0.$ \\
		
		If $k\geq 8,$ then $ \mu^*(f) \geq \frac{7}{9} k-\frac{34}{9}-\frac{5}{9} \times \frac{k}{2} = \frac{1}{2} k-\frac{34}{9} > 0.$
		

		
		It follows from  (1)-(10) that $\mu^*(x) \geq 0$ for each $x \in V(G) \cup F(G)$. 
		This contradiction completes the proof. 
		
		\noindent \textbf{Acknowledgments} The first author is (partially) supported by the Centre of Excellence in Mathematics, Ministry of Higher Education, Science, Research, and Innovation, Thailand.

	\end{document}